\documentclass{article}
\usepackage{amsmath,amssymb,amsfonts}
\usepackage{theorem}
\usepackage{color}
\usepackage{hyperref}
\usepackage{bbm}

\theoremstyle{change}
\newtheorem{definition}{Definition:}[section]
\newtheorem{proposition}[definition]{Proposition:}
\newtheorem{theorem}[definition]{Theorem:}
\newtheorem{lemma}[definition]{Lemma:}
\newtheorem{corollary}[definition]{Corollary:}
{\theorembodyfont{\rmfamily}
\newtheorem{remark}[definition]{Remark:}
}
{\theorembodyfont{\rmfamily}

}
\newenvironment{proof}
  {{\bf Proof:}}
  {\qquad \hspace*{\fill} $\Box$}

\begin{document}

\title{Topological entropy of a Lie group automorphism}
\author{V\'{\i}ctor Ayala\thanks{%
Supported by Proyecto Fondecyt $n^{o}$ 1150292, Conicyt} \ and\ Heriberto Rom%
\'{a}n-Flores \thanks{%
Supported by Proyecto Fondecyt $n^{o}$ 1151159 Conicyt} \\
%EndAName
Universidad de Tarapac\'a\\
Instituto de Alta Investigaci\'on\\
Phone 56 58 2205995. Casilla 7D, Arica, Chile\\
and\\
Adriano Da Silva \thanks{%
Supported by Fapesp grant n%
%TCIMACRO{\U{ba} }%
%BeginExpansion
${{}^o}$
%EndExpansion
2016/11135-2}\\
Instituto de Matem\'atica,\\
Universidade Estadual de Campinas\\
Cx. Postal 6065, 13.081-970 Campinas-SP, Brasil.\\
}
\date{\today}
\maketitle

\textbf{Abstract} Let $\varphi $ be an automorphism on a connected Lie group 
$G.$ Through several $G$-subgroups associated to the dynamics of $\varphi $
we analyze their topological entropy. Assume that $G$ belongs to the class
of finite semisimple center Lie groups which admits a $\varphi $ invariant
Levi subgroup. Then we prove that the topological entropy information of $%
\varphi $ is contained in the toral component of the unstable subgroup of $%
\varphi $ in the radical of $G$. We specialize the main result in a couple
of interesting situations.



\section{Introduction}

Let $G$ be a connected Lie group with Lie algebra $\mathfrak{g}.$ In \cite%
{AR-FDS} and \cite{Sil2} was shown that associated to a given endomorphism $%
\varphi $ on a connected Lie group $G$ there are some subgroups of $G$ that
are intrinsically connected with the dynamic behavior of $\varphi .$ We use
this set up to analyze the concept of topological entropy introduced by
Caldas and Patr\~{a}o in \cite{Pat}. In fact, in \cite{Pat2} the authors
consider the case where the state space is a connected Lie group. They
analyze the topological entropy of proper endomorphisms on different classes
of Lie groups. When the state space is a nilpotent or linear reductive Lie
group they show that the topological entropy of a proper endomorphism of $G$
coincides with the topological entropy of its restriction to $T(G)$. Here, $%
T(G)$ is the maximal connected compact subgroup of the center $Z_{G}$ of $G$
called the toral subgroup of $G$. By using the dynamic subgroups associated
with any automorphism $\varphi $ we extend and refine the results in \cite%
{Pat2}. In fact, assume that $G$ belongs to the class of finite semisimple
center Lie groups which admits a $\varphi $ invariant Levi subgroup.
Therefore, the topological entropy information of $\varphi $ coincides with
the topological entropy of its restriction to the toral component of the
unstable subgroup of $\varphi $ in $R$. Where $R$ denotes the radical of $G$%
. In particular, when $G$ is nilpotent or reductive Lie group in the
Harish-Chandra class, we prove that the topological entropy of $\varphi $
coincides with the topological entropy of its restriction to the unstable
subgroup of $T(G)$. Actually the information given by the entropy is
contained in a subgroup smaller than $T(G)$. We also show that the condition
on the existence of invariant Levi components is far from be trivial.

The paper is structured as follows: In Section 2 we establish the result
appear in \cite{AR-FDS} relatives to $\mathfrak{g}$-subalgebras and $G$%
-subgroups containing the dynamic behavior of a Lie algebra endomorphism and
its consequences on the dynamic of the associated Lie group endomorphism. In
Section 3 we use the dynamic subgroups to analyze the topological entropy of
an automorphism. Then, we show our main results and its consequences in some
particular cases.

\section{The dynamic of a Lie endomorphism}

For general fact about Lie theory we refer to \cite{Hel} and \cite{SM1}. Let 
$G$ be a connected Lie group of dimension $d$ with Lie algebra $\mathfrak{g}$
over a closed field$.$ For a given Lie groups $G$ and $H$ a continuous map $%
\varphi :G\rightarrow H$ is said to be a \textbf{homomorphism} if it
preserves the group structure. That is, $\varphi (gh)=\varphi (g)\varphi (h)$
for any $g,h\in G.$ If $G=H$ such map is said to be an \textbf{endomorphism }%
of $G$. Consider an endomorphism $\varphi :G\rightarrow G$ and denote by $%
\phi =(d\varphi )_{e\text{ }}:\mathfrak{g\rightarrow g}$ the corresponding
Lie algebra endomorphism, where as usual $d\varphi $ denotes de derivative
of $\varphi $. That is, $\phi $ is a linear map satisfying $\phi \lbrack
X,Y]=[\phi X,\phi Y]$ for any $X,Y\in \mathfrak{g}$. Here, $e$ denotes the
identity element of $G$.

This section is dedicated to establish some results about Lie algebra
endomorphism and its consequence on Lie groups endomorphism, appears in \cite%
{AR-FDS}.

First, for any eigenvalue $\alpha $ of $\phi $ let us consider its
generalized eigenspace 
\begin{equation*}
\mathfrak{g}_{\alpha }=\{X\in \mathfrak{g}:\; \;(\phi -\alpha )^{n}X=0,\; \;%
\mbox{ for some }\;n\geq 1\}.
\end{equation*}%
If $\beta $ is also an eigenvalue of $\phi $ then $[\mathfrak{g}_{\alpha },%
\mathfrak{g}_{\beta }]\subset \mathfrak{g}_{\alpha \beta },$ where $%
\mathfrak{g}_{\alpha \beta }=\{0\}$ if $\alpha \beta $ is not an eigenvalue
of $\phi $. Associate to $\phi $ there are several Lie subalgebras that are
intrinsically connected with its dynamic. In fact, let us define the
following subsets of $\mathfrak{g}$ for any arbitrary $\phi $-eigenvalue $%
\alpha :$ 
\begin{equation*}
\mathfrak{g}_{\phi }=\bigoplus_{\alpha \neq 0}\mathfrak{g}_{\alpha },\; \;
\; \; \mathfrak{k}_{\phi }=\ker (\phi ^{d}),\text{ }\mathfrak{g}%
^{+}=\bigoplus_{|\alpha |>1}\mathfrak{g}_{\alpha },\; \; \; \; \mathfrak{g}%
^{0}=\bigoplus_{\alpha :|\alpha |=1}\mathfrak{g}_{\alpha }\; \; \; \; 
\mathfrak{g}{-}=\bigoplus_{0<|\alpha |<1}\mathfrak{g}_{\alpha },
\end{equation*}%
\begin{equation*}
\text{ \ }\mathfrak{g}^{+,0}=\mathfrak{g}^{+}\oplus \mathfrak{g}^{0}\text{
and }\; \mathfrak{g}^{-,0}=\mathfrak{g}^{-}\oplus \mathfrak{g}^{0}.
\end{equation*}

We denote also $\mathfrak{g}_{\phi }=\mathfrak{g}^{+}\oplus \mathfrak{g}%
^{0}\oplus \mathfrak{g}^{-}$ and $\mathfrak{g}=\mathfrak{g}_{\phi }\oplus 
\mathfrak{k}_{\phi }$. It turns out that all these subspaces are Lie
subalgebras. Moreover, $\mathfrak{g}^{+}$ and $\mathfrak{g}^{-}$ are
nilpotent. If $\mathfrak{g}$ is a real Lie algebra, the algebras above are
also well defined.

\begin{remark}
\label{remark} We should note that in both cases, i.e., when $\mathfrak{g}$
is real or complex, the restriction of $\phi |_{\mathfrak{g}_{\phi }}$ is an
automorphism of $\mathfrak{g}_{\phi }$. Furthermore, the restriction of $%
\phi $ to the Lie subalgebras $\mathfrak{g}^{+}$, $\mathfrak{g}^{0}$ and $%
\mathfrak{g}^{-}$ satisfies the inequalities 
\begin{equation*}
|\phi ^{m}(X)|\text{ }\geq c\mu ^{-m}|X|\; \mbox{ for any }\;X\in \mathfrak{g%
}^{+}\text{ and }m\in \mathbb{N},
\end{equation*}%
and 
\begin{equation*}
|\phi ^{m}(Y)|\text{ }\leq c^{-1}\mu ^{m}|Y|\; \mbox{ for any }\;Y\in 
\mathfrak{g}^{-}\text{ and }m\in \mathbb{N},
\end{equation*}%
for some $c\geq 1$ and $\mu \in (0,1)$. Actually, more is true, 
\begin{equation*}
\text{ for any }a>0\text{ and }Z\in \mathfrak{g}^{0}\text{ it holds that }%
|\phi ^{m}(Z)|\, \mu ^{a|m|}\rightarrow 0\; \mbox{ as }m\rightarrow \pm
\infty .
\end{equation*}
\end{remark}

In the sequel any Lie group will be real. A Lie subgroup $H\subset G$ is
said to be $\varphi $-invariant if $\varphi (H)\subset H$.

The \textbf{dynamic subgroups} of $G$ induced by $\varphi $ are the Lie
subgroups, $G_{\varphi }$, $K_{\varphi }$, $G^{+}$, $G^{0}$, $G^{-}$, $%
G^{+,0}$ and $G^{-,0}$ associated with the Lie subalgebras $\mathfrak{g}%
_{\phi }$, $\mathfrak{k}_{\phi }$, $\mathfrak{g}^{+}$, $\mathfrak{g}^{0}$, $%
\mathfrak{g}^{-}$, $\mathfrak{g}^{+,0}$ and $\mathfrak{g}^{-,0}$,
respectively. The subgroups $G^{+}$, $G^{0}$ and $G^{-}$ are called the 
\textbf{unstable}, \textbf{central} and \textbf{stable} subgroups of $%
\varphi $ in $G,$ respectively. Before to give in Proposition 2.3,
Proposition 2.4 and Theorem 2.6 the main properties of these subgroups
proved in \cite{AR-FDS}, let us state a very special topological property of
Lie subgroups that will also need in the next sections, \cite{AR-FDS}.

\begin{lemma}
\label{closed} Let $G$ be a Lie group with Lie algebra $\mathfrak{g.}$ And, $%
H$ and $K$ Lie subgroups of $G$ with Lie algebras $\mathfrak{h}$ and $%
\mathfrak{k}$, respectively such that $\mathfrak{h}\oplus \mathfrak{k}=%
\mathfrak{g}$. Then, 
\begin{equation*}
H\text{ and }K\text{ are closed}\Leftrightarrow H\cap K\text{ is a discrete
subgroup.}
\end{equation*}
\end{lemma}

\begin{proposition}
\label{subgroups} It holds:

\begin{itemize}
\item[1.] All the dynamic subgroups are $\varphi $-invariant

\item[2.] The subgroup $K_{\varphi }=\ker (\varphi ^{d})_{0}$ is normal.
Moreover, 
\begin{equation*}
G=G_{\varphi }K_{\varphi }\text{ \ and }G_{\varphi }=\mathrm{Im}(\varphi
^{d})
\end{equation*}

\item[3.] The restriction of $\varphi $ is expanding on $G^{+}$ and
contracting on $G^{-}$

\item[4.] If $G_{\varphi }$ is a solvable Lie group therefore 
\begin{equation}
G_{\varphi }=G^{+,0}G^{-}=G^{-,0}G^{+}=G^{+}G^{0}G^{-}  \label{decomposition}
\end{equation}

\item[5.] If $G_{\varphi }$ is semisimple and $G^{0}$ is compact, then $%
G_{\varphi }=G^{0}$. In particular, if $G$ is any connected Lie group such
that $G^{0}$ is compact, then $G_{\varphi }$ has also the decomposition (\ref%
{decomposition}).
\end{itemize}
\end{proposition}

\begin{definition}
Let $\varphi $ be an endomorphism of the Lie group $G.$ We say that $\varphi 
$ \textbf{decomposes} $G$ if $G_{\varphi }$ satisfy (\ref{decomposition}),
\end{definition}

\begin{proposition}
\label{automorphism} Assume that $\varphi $ restricted to $G_{\varphi }$ is
an automorphism in the induced topology of $G$. Then,

\begin{itemize}
\item[1.] $G^{+,0}\cap G^{-}=G^{+}\cap G^{-}=G^{0}\cap G^{-}=G^{-,0}\cap
G^{+}=G^{+}\cap G^{0}=\{e\}$

\item[2.] The dynamic subgroups induced by $\varphi $ are closed in $G$

\item[3.] For $n\geq d$, $\ker (\varphi ^{n})=K_{\varphi }$. In particular, $%
\ker (\varphi ^{n})$ is connected$.$

\item[4.] If $G$ is simply connected then $G_{\varphi }$ and $K_{\varphi }$
are simply connected. Moreover, the restriction of $\varphi $ to $G_{\varphi
}$ is an automorphism. And, any subgroup induced by an endomorphism $\varphi 
$ of $G$ is closed..
\end{itemize}
\end{proposition}

The main result in \cite{AR-FDS} shows that the unstable/stable subgroup of
a compact $\varphi $-invariant subgroup of $G_{\varphi }$ is contained in
its center. In particular, this implies the decomposition of the group when $%
G_{\varphi }$ is compact.

\begin{theorem}
\label{compact} Let $G$ be a Lie group and $\varphi $ an endomorphism of $G$%
. If $H\subset G_{\varphi }$ is a $\varphi $-invariant compact subgroup,
then $H^{+}$ and $H^{-}$ are contained in the center $Z_{H}$ of $H$. In
particular, if $G_{\varphi }$ is compact the group $G$ is decomposable.
Furthermore, assume that $G$ is solvable. Therefore, if $\varphi
|_{G_{\varphi }}$ is also an automorphism it follows that any fixed point of 
$\varphi $ is contained in $G^{0}$.
\end{theorem}

\section{Topological entropy}

In this section we show how the dynamic subgroups induced by any
endomorphism can be used in order to get information about their dynamic
behavior. We analyze the concept of topological entropy for continuous
proper maps on locally compact separable metric spaces as introduced by
Caldas and Patr\~{a}o in \cite{Pat}.

Let $X$ be a topological space and $\xi :X\rightarrow X$ a proper map, that
is, $\xi $ is a continuous map with the property that its pre-image of any
compact set is still a compact set. An \textbf{admissible} cover of $X$ is a
finite open cover $\Psi $ such that any element $A\in \Psi $ has a compact
closure or its complement is a compact set. For a given admissible cover $%
\Psi $ of $X$ and any $n\in \mathbb{N}$ the set 
\begin{equation*}
\Psi ^{n}:=\{A_{0}\cap \xi ^{-1}(A_{1})\cap \cdots \cap \xi ^{-n}(A_{n}),\;
\;A_{i}\in \Psi \}
\end{equation*}%
is also an admissible cover of $X$. We denote by $N(\Psi ^{n})$ the smallest
cardinality of all subcovers of $\Psi ^{n}$. The \textbf{topological entropy}
of $\xi $ is then defined by 
\begin{equation*}
h_{\mathrm{top}}\left( \xi \right) :=\sup_{\Psi }h_{\mathrm{top}}\left( \xi
,\Psi \right) ,\; \; \mbox{ where }\; \; \;h_{\mathrm{top}}\left( \xi ,\Psi
\right) :=\lim_{n\rightarrow \infty }\frac{1}{n}\log N(\Psi ^{n})
\end{equation*}%
and $\Psi $ varies among all the admissible subcovers of $X$.

It is worth to notice that when $X$ is a compact space, the $\mathrm{top}$%
-entropy definition coincide with the Adler-Konheim-MacAndrew topological
entropy notion, \cite{AKM}.

For continuous maps on metric spaces, we have also the concept of entropy
introduced by Bowen in \cite{RB} defined as follows: Let $(X, d)$ be a
metric space and $\xi:X\rightarrow X$ a continuous map. Given a subset $%
Y\subset X$, $\varepsilon>0$ and $n\in \mathbb{N}$ we say that $S\subset X$
is an $(n, \varepsilon)$-spanning set of $Y$ if for every $y\in Y$ there
exists $x\in S$ such that 
\begin{equation*}
d\left(\xi^i\left(x\right), \xi^i\left(y\right)\right)<\varepsilon \; \; 
\mbox{
for } \; \;0\leq i\leq n.
\end{equation*}
We denote by $s(n, \varepsilon, Y)$ the smallest cardinality of any $(n,
\varepsilon)$-spanning set of $Y$ and define 
\begin{equation*}
s(\varepsilon, Y):=\lim_{n\rightarrow \infty}\frac{1}{n}\log s(n,
\varepsilon, Y)\; \mbox{ and }\; \; h(\xi, Y):=\lim_{\varepsilon \rightarrow
0^+}s(\varepsilon, Y).
\end{equation*}

The \textbf{Bowen entropy} of $\xi $ with respect to the metric $d$ is
defined by 
\begin{equation*}
h_{d}\left( \xi \right) :=\sup_{K}h\left( \xi ,K\right)
\end{equation*}%
where $K$ varies among the compact subsets of $X$.

We also mention the concept of \textbf{$d$-entropy} defined by 
\begin{equation*}
h^{d}\left( \xi \right) :=\sup_{Y}h\left( \xi ,Y\right) .
\end{equation*}%
In this case $Y$ varies among all the subsets of $X$. It holds that 
\begin{equation*}
h_{d}\left( \xi \right) \leq h^{d}\left( \xi \right) .
\end{equation*}

Next, we introduce the concept of an admissible metric. The metric $d$ is 
\textbf{admissible} if it satisfies the following conditions:

\begin{itemize}
\item[(i)] If $\Psi _{\delta }:=\{B_{\delta }(x_{1}),\ldots ,B_{\delta
}(x_{n})\}$ is a cover of $X$ then for every $\delta \in (a,b)$ with $0<a<b$%
, there exists $\delta _{0}\in (a,b)$ such that $\Psi _{\delta _{0}}$ is
admissible;

\item[(ii)] Every admissible cover of $X$ has a Lebesgue number.
\end{itemize}

The following result from \cite{Pat}, Proposition 2.2, show that for an
admissible metrics the concept of $d$-entropy and topological entropy
coincides.

\begin{proposition}
Let $\xi$ be a continuous map on the metric space $(X, d)$. If the metric $d$
is admissible, then 
\begin{equation*}
h_{\mathrm{top}}\left(\xi \right)=h^d\left(\xi \right).
\end{equation*}
\end{proposition}

From here every topological space under consideration will be a locally
compact separable metric spaces. Let $X$ be a topological space and denote
by $X_{\infty }$ the one point compactification of $X$. We know that $%
X_{\infty }$ is defined as the disjoint union of $X$ with $\{ \infty \}$
where $\infty $ is some point that does not belongs to $X$ called the 
\textbf{point at the infinite}. The topology in $X_{\infty }$ consists by
the former open sets in $X$ and by the sets $U$ $\cup $ $\{ \infty \}$,
where the complement of $U$ in $X$ is compact. Let $Y$ be a topological
space and $\xi :X\rightarrow Y$ a proper map. Define $\widetilde{\xi }%
:X_{\infty }\rightarrow Y_{\infty }$ by 
\begin{equation*}
\widetilde{\xi }(x)=\left \{ 
\begin{array}{cc}
\xi (x), & x\neq \infty _{X} \\ 
\infty _{Y} & x=\infty _{X}%
\end{array}%
\right.
\end{equation*}%
then, $\widetilde{\xi }$ is a continuous map called the extension of $\xi $.
The following result (Proposition 2.3 of \cite{Pat}) assures the existence
of an admissible metrics on the topological spaces that we are considering
in here.

\begin{proposition}
Let $\xi $ be a proper map of $X$ and consider the metric $d$ in $X$
obtained by the restriction of some metric $\widetilde{d}$ on $X_{\infty }$.
It turns out that $d$ is an admissible metric and 
\begin{equation*}
h^{d}(\xi )=h^{\widetilde{d}}\left( \widetilde{\xi }\, \right) ,
\end{equation*}%
where $\widetilde{\xi }$ is the extension of $\xi $. In particular, $h_{%
\mathrm{top}}\left( \xi \right) =h_{\mathrm{top}}\left( \widetilde{\xi }\,
\right) $.
\end{proposition}

Let $Y\subset X$ be a closed subspace of $X$. By taking the closure of $Y$
in $X_{\infty }$ we obtain that $\mathrm{cl}_{X_{\infty }}(Y)=Y_{\infty }$. By
Proposition 3.2, it follows that the restriction $d_{Y}$ of an admissible
metric $d$ to $Y$ is an admissible metric in $Y$. Moreover, if $\xi $ is a
proper map of $X$ and $Y$ is $\xi $-invariant, we have that $\widetilde{\xi }%
|_{Y_{\infty }}=\widetilde{\xi |_{Y}}$. Therefore, we get the following
inequality.

\begin{proposition}
\label{lower} If $\xi:X\rightarrow X$ is a proper map and $Y\subset X$ is a
closed $\xi$-invariant subspace, then 
\begin{equation*}
h_{\mathrm{top}}\left(\xi|_{Y}\right)\leq h_{\mathrm{top}}\left(\xi \right).
\end{equation*}
\end{proposition}

\begin{proof}
By considering the extension $\widetilde{\xi}$ of $\xi$ and an admissible
metric $d$ on $X$, given by the restriction of some metric $\widetilde{d}$
on $X_{\infty}$, we have that 
\begin{equation*}
h_{\mathrm{top}}(\xi)=h^d(\xi)=h^{\widetilde{d}}\left(\widetilde{\xi}\,
\right)=h_{\widetilde{d}}\left(\widetilde{\xi}\right) \geq h_{\widetilde{d}%
}\left(\widetilde{\xi}, Y_{\infty}\right)=h_{\widetilde{d}%
_{Y_{\infty}}}\left(\widetilde{\xi}|_{Y_{\infty}}\right)
\end{equation*}
\begin{equation*}
=h_{\widetilde{d_Y}}\left(\widetilde{\xi|_{Y}}\right)=h^{\widetilde{d_Y}%
}\left(\widetilde{\xi|_{Y}}\right)=h^{d_Y}\left(\xi|_{Y}\right)=h_{\mathrm{%
top}}\left(\xi|_Y\right)
\end{equation*}
as stated.
\end{proof}

The proof of the next two propositions can be found in \cite{Pat2},
Proposition 2.2 and \cite{Pat}, Proposition 2.1, respectively.

\begin{proposition}
\label{product} Let $X$ and $Y$ be topological spaces and $\xi :X\rightarrow
X$, $\zeta :Y\rightarrow Y$ a proper map. Then 
\begin{equation*}
h_{\mathrm{top}}\left( \xi \times \zeta \right) =h_{\mathrm{top}}\left( \xi
\right) +h_{\mathrm{top}}\left( \zeta \right) .
\end{equation*}
\end{proposition}

\begin{proposition}
\label{projection} Let $X$ and $Y$ be topological spaces, $\xi :X\rightarrow
X$ and $\zeta :Y\rightarrow Y$ two proper maps. If $f:X\rightarrow Y$ is a
proper surjective map such that $f\circ \xi =\zeta \circ f$, then $h_{%
\mathrm{top}}(\xi )\geq h_{\mathrm{top}}(\zeta )$.
\end{proposition}

\subsection{Topological entropy of automorphisms}

In this section we show that all the information given by the topological
entropy of an automorphism $\varphi $ of $G$ is contained in the toral part
of the unstable subgroup of $\varphi $ in the radical of $G$.

\begin{definition}
Let $G$ be a Lie group and denote by $Z_{G}$ its center. The \textbf{toral}
component $T(G)$ of $G$ is the greatest compact connected Lie subgroup of $%
Z_G$.
\end{definition}

We start by giving a lower bound for the topological entropy.

\begin{proposition}
\label{lowertoral} If $\varphi$ is an automorphism of $G$, then 
\begin{equation*}
h_{\mathrm{top}}(\varphi)\geq h_{\mathrm{top}}\left(\varphi|_{T(G^+)}\right)
\end{equation*}
where $T(G^+)$ is the toral component of the subgroup $G^+$.
\end{proposition}

\begin{proof}
Since $\varphi $ is an automorphism, $G^{+}$ is a closed subgroup of $G$. By
Proposition \ref{lower} we know that $h_{\mathrm{top}}(\varphi )\geq h_{%
\mathrm{top}}(\varphi |_{G^{+}})$. Furthermore, $G^{+}$ is a connected
nilpotent Lie group and $\varphi |_{G^{+}}$ is an automorphism, so Theorem
4.3 of \cite{Pat2} assures that $h_{\mathrm{top}}(\varphi |_{G^{+}})=h_{%
\mathrm{top}}\left( \varphi |_{T(G^{+})}\right) $. Therefore, 
\begin{equation*}
h_{\mathrm{top}}(\varphi )\geq h_{\mathrm{top}}(\varphi |_{G^{+}})=h_{%
\mathrm{top}}\left( \varphi |_{T(G^{+})}\right) .
\end{equation*}
\end{proof}

As a consequence of Proposition 3.7 we obtain: in order to check if the
topological entropy of any automorphism of $G$ is not zero, it is enough to
restrict it to a considerable smaller subgroup of $G$. Actually, in the next
sections we show that for a vast class of Lie groups the topological entropy
of an automorphism is completely determined by its restriction to the toral
component of the unstable subgroup of the radical of $G$.

Furthermore, for an automorphism which turn out $G$ decomposable, we get

\begin{theorem}
\label{solvable} Let $G$ be a connected Lie group and $\varphi$ an
automorphism. If $G$ is decomposable, then 
\begin{equation*}
h_{\mathrm{top}}\left(\varphi \right)= h_{\mathrm{top}}\left(%
\varphi|_{T(G^+)}\right)
\end{equation*}
\end{theorem}

\begin{proof}
By Proposition \ref{automorphism} we obtain 
\begin{equation*}
G^{+}\cap G^{0}=G^{+}\cap G^{-}=G^{-}\cap G^{0}=\{e\}
\end{equation*}
Since $G$ is decomposable, the map $f:G^{+}\times G^{0}\times
G^{-}\rightarrow G$ is a homeomorphism that commutates $\varphi $ and $%
\varphi |_{G^{+}}\times \varphi |_{G^{0}}\times \varphi |_{G^{-}}$.
Therefore, by Propositions \ref{product} and \ref{projection} we get that 
\begin{equation*}
h_{\mathrm{top}}(\varphi )\leq h_{\mathrm{top}}\left( \varphi
|_{G^{+}}\right) +h_{\mathrm{top}}\left( \varphi |_{G^{0}}\right) +h_{%
\mathrm{top}}\left( \varphi |_{G^{-}}\right) .
\end{equation*}%
Since $\varphi |_{G^{-}}$ and $\varphi |_{G^{0}}$ have only eigenvalues with
modulo smaller or equal to 1, Corollary 16 of \cite{RB} and Theorem 3.2 of 
\cite{Pat} implies that $h_{\mathrm{top}}(\varphi |_{G^{-}})=h_{\mathrm{top}%
}(\varphi |_{G^{0}})=0$ and therefore $h_{\mathrm{top}}(\varphi )\leq h_{%
\mathrm{top}}\left( \varphi |_{G^{+}}\right) $. Moreover, since $G^{+}$ is a
connected nilpotent Lie group, we obtain by Theorem 4.3 of \cite{Pat2} that $%
h_{\mathrm{top}}\left( \varphi |_{G^{+}}\right) =h_{\mathrm{top}}\left(
\varphi |_{T(G^{+})}\right) $ which, together with Proposition \ref%
{lowertoral} gives us 
\begin{equation*}
h_{\mathrm{top}}(\varphi )=h_{\mathrm{top}}\left( \varphi |_{T(G^{+})}\right)
\end{equation*}%
as we wish to prove.
\end{proof}

\begin{corollary}
If $\varphi $ is an automorphism of a solvable Lie group $G$ it follows that 
\begin{equation*}
h_{\mathrm{top}}(\varphi )=h_{\mathrm{top}}\left( \varphi
|_{T(G^{+})}\right) .
\end{equation*}%
In particular, when $G$ is nilpotent, we get $T(G^{+})=T(G)^{+}$.
\end{corollary}

\begin{proof}
By Proposition \ref{subgroups}, any solvable Lie group $G$ is decomposable
and the result follows from Theorem \ref{solvable} above. Let us assume now
that $G$ is a nilpotent Lie group. Since $T(G)^{+}$ is a compact subgroup of 
$Z_{G^{+}}$ the inclusion $T(G)^{+}\subset T(G^{+})$ always hold, and we
only have to show that $T(G^{+})\subset T(G)^{+}$. By Proposition 4.1 of 
\cite{Pat2}, the nilpotent Lie group $G/T(G)$ is simply connected and so the
compact subgroup $\pi (T(G^{+}))$ has to be trivial, where 
\begin{equation*}
\pi :G\rightarrow G/T(G)
\end{equation*}%
is the canonical projection. Consequently, $T(G^{+})\subset \ker (\pi )=T(G)$
showing that $T(G^{+})\subset T(G)^{+}$ and concluding the proof.
\end{proof}

\begin{corollary}
If $\varphi $ is an automorphism of a Lie group $G$ and $G^{0}$ is a compact
subgroup, then it follows that 
\begin{equation*}
h_{\mathrm{top}}(\varphi )=h_{\mathrm{top}}\left( \varphi
|_{T(G^{+})}\right) .
\end{equation*}%
In particular, if $G$ is a connected and compact Lie group $%
T(G^{+})=T(G)^{+} $.
\end{corollary}

\begin{proof}
The first assertion follows from the assumption $G$ is decomposable. In
addition, if $G$ is also compact, Proposition \ref{compact} implies $%
G^{+}\subset Z_{G}$ showing that $T(G^{+})=T(G)^{+}$ as stated.
\end{proof}

\begin{corollary}
If $G$ is a simply connected solvable Lie group and $\varphi$ an
automorphism, then 
\begin{equation*}
h_{\mathrm{top}}(\varphi)=0.
\end{equation*}
\end{corollary}

\begin{proof}
In fact, it is well known that any connected subgroup of a simply connected
solvable Lie group is simply connected, \cite{SM2},\cite{Hel}. Therefore, $%
G^{+}$ is simply connected and consequently $T(G^{+})=\{e\}$.
\end{proof}

\subsection{Lie groups with finite semisimple center}

In \cite{ADS} the authors introduce the notion of finite semisimple center
Lie group $G$ to study the controllability property of a linear control
systems on $G,$ \cite{AyTi}. In this section we compute the entropy of any
automorphisms $\varphi $\ for this special class of Lie groups. If $G$ admit
a $\varphi $-invariant Levi component, their topological entropy depends
just on the unstable subgroup of the radical of $G$.

\begin{definition}
Let $G$ be a connected Lie group. We say that $G$ has \textbf{finite
semisimple center} if any semisimple subgroup of $G$ has finite center.
\end{definition}

We should notice that there are many groups with such property. For
instance, solvable Lie groups and semisimple Lie groups with finite center
and also their semi-direct products. Moreover, any reductive Lie group in
the Harish-Chandra class has also finite semisimple center.

For a given Lie algebra $\mathfrak{g}$, a \textbf{Levi subalgebra }of $%
\mathfrak{g}$ is a maximal semisimple Lie subalgebra $\mathfrak{s}$ of $%
\mathfrak{g}$, in the sense, that for any semisimple Lie subalgebra $%
\mathfrak{l}$ of $\mathfrak{g}$, there exists and inner automorphism $\psi $
such that $\psi (\mathfrak{l})\subset \mathfrak{s}$. By Theorem 4.1 of \cite%
{ALEB}, any Lie algebra admits a Levi subalgebra. If $G$ is a connected Lie
group with Lie algebra $\mathfrak{g}$ a \textbf{Levi subgroup} of $G$ is a
connected semisimple Lie group $S\subset G$ such that its Lie algebra $%
\mathfrak{s}\subset \mathfrak{g}$ is a Levi subalgebra.

Let us denote by $\mathfrak{r}$ for the radical of the Lie algebra $%
\mathfrak{g}$ and by $R$ the radical of $G$, that is, $R$ is the closed
connected Lie subgroup of $G$ with Lie algebra $\mathfrak{r}$. From Theorem
4.1 of \cite{ALEB} and its corollaries, we have that $G$ decomposes as $G=RS$
with $\dim(S\cap R)=0$. For groups with finite semisimple center we obtain.

\begin{proposition}
\label{invLevi} If $G$ is a connected Lie group with finite semisimple
center then $R\cap S$ is a discrete finite subgroup of $G$. Thus, any Levi
subgroup $S$ of $G$ is closed in $G$.
\end{proposition}

\begin{proof}
Since $S$ is connected, any discrete normal subgroup of $S$ is contained in
its center. However, $R$ is normal so $R\cap S$ is also a normal subgroup of 
$S.$ The result follows because $Z_{S}$ is finite by assumption. Moreover,
by Lemma \ref{closed}, $R\cap S$ is discrete if and only if $R$ and $S$ are
closed subgroups of $G$. Since any Levi subgroup is conjugated to $S$ (see
Corollary 3, Chapter I of \cite{ALEB}) the prove is finish.
\end{proof}

Now, we are in a position to establish and prove our main result about the
topological entropy of automorphisms.

\begin{theorem}
\label{main} Let $G$ be a Lie group and $\varphi $ a $G$-automorphism. If $G$
has finite semisimple center and admits a $\varphi $ invariant Levi
subgroup, then 
\begin{equation*}
h_{\mathrm{top}}(\varphi )=h_{\mathrm{top}}\left( \varphi |_{T\left(
R^{+}\right) }\right) ,
\end{equation*}%
where $R^{+}$ is the corresponding unstable component of the radical $R$ of $%
G$.
\end{theorem}

\begin{proof}
Let $S$ be the $\varphi $-invariant Levi subgroup of $G$. Since $\varphi $
is an automorphism we know that $R$ is also $\varphi $-invariant. Since $G$
has the semisimple finite center property, Proposition \ref{invLevi} implies
that $S$ is a closed subgroup and therefore $R\cap S$ is a discrete finite
subgroup of $G$.

Let us consider the product map $f:R\times S\rightarrow G$ defined by $%
(r,s)\mapsto rs$. By Malcev's Theorem (see \cite{ALEB}, Theorem 4.1) we get
that $f$ is surjective. On the other hand, since $\varphi $ is an
automorphism it follows that $f$ commutates with $\varphi $ and $\varphi
|_{R}\times \varphi |_{S}$. Moreover, it holds:

\begin{itemize}
\item[1.] $f$ is a continuous closed map.

In fact, if $r_{n}\rightarrow r$ in $R$ and $s_{n}\rightarrow s$ in $S$ then 
$r_{n}s_{n}\rightarrow rs$ in $G$, showing that $f$ is continuous. Let us
consider two closed subsets $A\subset R$ and $B\subset S$ and assume that $%
x_{n}=r_{n}s_{n}\rightarrow rs$ with $r_{n}\in A$, $s_{n}\in B$, $r\in R$
and $s\in S$. By Lemma 6.14 of \cite{SM2} there exist neighborhoods $0\in
V\subset \mathfrak{r}$, $0\in U\subset \mathfrak{s}$ and $e\in W\subset G$
such that $(X,Y)\mapsto \mathrm{e}^{X}\mathrm{e}^{Y}\in W$ is a
diffeomorphism. Since $r_{n}s_{n}\rightarrow rs$, there exists $n_{0}\in 
\mathbb{N}$ such that $r^{-1}r_{n}s_{n}s^{-1}\in W$ for $n\geq n_{0}$ and
so, there are sequences $(v_{n})\subset \mathrm{e}^{V}$, $(u_{n})\subset 
\mathrm{e}^{U}$ such that $v_{n},u_{n}\rightarrow e$ and $%
r^{-1}r_{n}s_{n}s^{-1}=v_{n}u_{n}$ for $n\geq n_{0}$. Therefore, 
\begin{equation*}
v_{n}^{-1}r^{-1}r_{n}=u_{n}ss_{n}^{-1}\in R\cap S.
\end{equation*}
However, $R\cap S$ is a finite discrete subgroup. Thus, without lost of
generality we can assume that $v_{n}^{-1}r^{-1}r_{n}=u_{n}ss_{n}^{-1}=l$ for
some $l\in R\cap S$. Consequently, $r_{n}=rv_{n}l\rightarrow rl$ and $%
s_{n}=l^{-1}u_{n}s\rightarrow l^{-1}s$. Since $A$ and $B$ are closed
subsets, we must have $rl\in A$ and $l^{-1}s\in B$ implying that $%
rs=rll^{-1}s\in AB$ and showing that $f$ is in fact a closed map.

\item[2.] $f$ is a proper map.

Since $f$ is continuous and closed we just need to show that for any $x\in G$
the set $f^{-1}(x)$ is compact. However, using that $x=rs$ for some $r\in R$
and $s\in S$, it is straightforward to check that 
\begin{equation*}
f^{-1}(x)=\{(rl,l^{-1}s),\;l\in R\cap S\},
\end{equation*}%
which shows that $f^{-1}(x)$ is in fact a finite subset of $R\times S$.
\end{itemize}

By Propositions \ref{lower}, \ref{product}, \ref{projection} we get 
\begin{equation*}
h_{\mathrm{top}}\left( \varphi |_{R}\right) +h_{\mathrm{top}}\left( \varphi
|_{S}\right) =h_{\mathrm{top}}\left( \varphi |_{R}\times \varphi
|_{S}\right) \geq h_{\mathrm{top}}(\varphi |_{G})\geq h_{\mathrm{top}%
}(\varphi |_{R}).
\end{equation*}%
By Theorem 5.3 of \cite{Pat2} the entropy of any surjective endomorphism of
a semisimple Lie group is zero. Therefore, $h_{\mathrm{top}}\left( \varphi
|_{S}\right) =0$ which combined with Theorem 3.8 give us 
\begin{equation*}
h_{\mathrm{top}}(\varphi )=h_{\mathrm{top}}\left( \varphi |_{R}\right) =h_{%
\mathrm{top}}\left( \varphi |_{T\left( R^{+}\right) }\right)
\end{equation*}%
as we wish to prove.
\end{proof}

In particular we obtain

\begin{corollary}
If $G$ is a reductive Lie group in the Harish-Chandra class, then the
entropy of any automorphism $\varphi$ of $G$ satisfies 
\begin{equation*}
h_{\mathrm{top}}(\varphi)=h_{\mathrm{top}}\left(\varphi|_{T\left(G\right)^+}%
\right).
\end{equation*}
\end{corollary}

\begin{proof}
In fact, if $G$ is a reductible Lie group in the Harish-Chandra class, we
know that $G=Z_{G}G^{\prime }$ where $G^{\prime }$ is the derived subgroup
of $G$. Moreover, $G^{\prime }$ is a maximal connected semisimple Lie
subgroup of $G$ and has finite center. Of course, since $G^{\prime }$ is
certainly $\varphi $-invariant we get from our Theorem 3.14 that 
\begin{equation*}
h_{\mathrm{top}}(\varphi )=h_{\mathrm{top}}\left( \varphi |_{T\left(
Z_{G}^{+}\right) }\right) .
\end{equation*}%
However, $T(Z_{G}^{+})$ is a compact connected subgroup in $Z_{G}\cap
G^{+}\subset Z_{G^{+}}$ showing that $T(Z_{G}^{+})\subset T(G)^{+}$. Now,
since $T(Z_{G}^{+})\subset Z_{G}$ we must have that $T(Z_{G}^{+})\subset
T(G) $ and consequently $T(Z_{G}^{+})=T(G)^{+}$ concluding the proof.
\end{proof}

For a simply connected solvable Lie group $G$ we prove in Corollary 3.11
that $h_{\mathrm{top}}(\varphi )=0,$ for any $\varphi $ automorphism$.$It is
possible to extend this result as follows.

\begin{corollary}
Let $G$ be a simply connected Lie group with finite center and $\varphi$ an
automorphism of $G$. If $G$ admits a $\varphi$-invariant Levi subgroup then 
\begin{equation*}
h_{\mathrm{top}}(\varphi)=0.
\end{equation*}
\end{corollary}

In the sequel, we show that the condition on the existence of invariant Levi
components is far from be trivial.

\begin{definition}
We say that an automorphism $\varphi $ of $G$ is \textbf{semisimple} if its
differential at $e\in G$ is semisimple. That is, $(d\varphi )_{e}$ is
diagonalizable as a map of the complexification of $\mathfrak{g}$ .
\end{definition}

For this particular class of automorphisms we get.

\begin{corollary}
Let $G$ be a Lie group and $\varphi $ a semisimple automorphim of $G$. If $G$
has the finite semisimple center property, then 
\begin{equation*}
h_{\mathrm{top}}(\varphi )=h_{\mathrm{top}}\left( \varphi |_{T\left(
R^{+}\right) }\right) .
\end{equation*}
\end{corollary}

\begin{proof}
By Corollary 5.2 of \cite{Mos}, any group of semisimple automorphisms of a
Lie algebra leaves invariant a Levi subalgebra. Therefore, there exists a
Levi subalgebra $\mathfrak{s}\subset \mathfrak{g}$ that is invariant by the
group of semisimple automorphisms $\{ \phi ^{n};\; \;n\in \mathbb{Z}\}$.
Hence, the Levi subgroup $S\subset G$ with Lie algebra $\mathfrak{s}$ is $%
\varphi $-invariant and by Theorem \ref{main} we obtain 
\begin{equation*}
h_{\mathrm{top}}(\varphi )=h_{\mathrm{top}}\left( \varphi |_{T\left(
R^{+}\right) }\right) .
\end{equation*}
\end{proof}


\begin{thebibliography}{99}
\bibitem{AKM} \textsc{R. Adler, A. Konheim and H. MacAndrew}, \emph{%
Topological entropy}, Trans. Amer. Math. Soc. 114 (1965), pp.

\bibitem{ADS} \textsc{V. Ayala and A. Da silva, }\emph{Controllability of
linear systems on Lie groups with finite semisimple center, }Submitted to
SIAM Journal on Control and Optimization, (2016).

\bibitem{AR-FDS} V. \textsc{Ayala, H. Roman-Flores and A. Da Silva, }\emph{%
The dynamic of a Lie algebra endomorphism, }Submitted to Discrete and
Continuous Dynamical Systems, January 2017.

\bibitem{AyTi} \textsc{V. Ayala and J. Tirao}, \emph{Linear control systems
on Lie groups and Controllability}, Eds. G. Ferreyra et al., Amer. Math.
Soc., Providence, RI, 1999.

\bibitem{RB} \textsc{Bowen, R.}, \emph{Entropy for group endomorphisms and
homogeneous spaces}, Trans. Amer. Math. Soc., 153 (1971), pp. 401--414.

\bibitem{Pat2} \textsc{A. Caldas and M. Patrao}, \emph{Entropy of
Endomorphisms of Lie Groups}, Discrete and Continuous Dynamical Systems -
Series A , 33, \textbf{4} (2013), pp. 1351-1363.

%\bibitem {Pat3}\textsc{A. Caldas and M. Patr\~ao}, \emph{Entropy and its Variational Principle for Locally Compact Metrizable Systems}, to appear (2016).
%in Ergodic Theory and Dynamical Systems

\bibitem{Sil2} \textsc{A.~Da Silva}, \emph{\ Controllability of linear
systems on solvable Lie groups,} SIAM Journal on Control and Optimization,
54, \textbf{1} (2016), pp. 372-390.

%\bibitem{Fer} \textsc{T. Ferraiol, M. Patr˜ao and L. Seco}, \emph{Jordan decomposition and dynamics
%on flag manifolds}, Discrete Contin. Dyn. Syst. A, 26, {\bf 3} (2010), pp. 923-947.

%\bibitem{Han} \textsc{M. Handel and B. Kitchens}, \emph{Metrics and entropy for non-compact spaces}, Isr. J. Math., 91 (1995), pp. 253-271.

\bibitem{Hel} \textsc{S. Helgason}, \emph{Differential Geometry, Lie Groups
and Symmetric Spaces}, American Mathematical Society, Providence-Rhode
Island, (2001).

%\bibitem{CoSa} \textsc{F.~Colonius and A.~J.~Santana}; {\it Stability and Topological Conjugacy for Affine Differential Equations}, Bol.\ Soc.\ Parana.\ Mat.  26 (2009), pp.~141--151.

%\bibitem{CoSa2} \textsc{F.~Colonius and A.~J.~Santana}; {\it Topological conjugacy for affine-linear flows and control systems}, Communications on Pure and Applied Analysis 10, (2001), pp.~847--857.%

%\bibitem {DaAy2}\textsc{A.~Da Silva and V.~Ayala},  {\it Control sets of linear systems on Lie groups} submitted.

%\bibitem{KaHA} \textsc{A.~Katok and B.~Hasselblatt}; {\it Introduction to the Modern Theory of Dynamical Systems}, 1st Edition, Cambridge University Press, Cambridge, 1996.%

\bibitem{Mos} \textsc{G. D. Mostow}; \emph{Fully Reducible Subgroups of
Algebraic Groups}, American Journal of Mathematics, 78, \textbf{1}, (1956),
pp. 200-221.

\bibitem{Pat} \textsc{M. Patr\~{a}o}, \emph{Entropy and its Variational
Principle for Non-Compact Metric Spaces}, Ergodic Theory and Dynamical
Systems, 30 (2010), pp. 1529-1542.

%\bibitem{RoSaVe} \textsc{O.~G.~Rocio, A.~J.~Santana and M.~A.~Verdi}, {\it Semigroups of Affine Groups, Controllability of Affine Systems and Affine Bilinear Systems in {$\mathrm{S\lowercase{l}}(2,\mathbb{R})\rtimes\mathbb{R}^{2}$}}, SIAM J.\ Control Optim. 48 (2009), pp.~1080--1088.%

%\bibitem {SM1}\textsc{L.~A.~B.~San Martin}, {\it Algebras de Lie}, Second Edition, Editora Unicamp, (2010).

\bibitem{ALEB} \textsc{A. L. Onishchik and E. B. Vinberg}, \emph{Lie Groups
and Lie Algebras III - Structure of Lie Groups and Lie Algebras}, Berlin:
Springer (1990).

\bibitem{SM1} \textsc{L. A. B. San Martin}, \emph{Algebras de Lie}, Second
Edition, Editora Unicamp, (2010).

\bibitem{SM2} \textsc{L.~A.~B.~San Martin}, \textit{Grupos de Lie},
unpublished, (2014).
\end{thebibliography}
\end{document}